\documentclass{article}
\usepackage{amsmath}
\usepackage{amssymb}
\usepackage{amsfonts}
\usepackage{amsthm}

\def\rank{\mathop{\rm rank}\nolimits}

 \newtheorem{thm}{Theorem}[section]
 \newtheorem{cor}[thm]{Corollary}
 \newtheorem{lem}[thm]{Lemma}
 \newtheorem{prop}[thm]{Proposition}
 \theoremstyle{definition}
 
 \theoremstyle{remark}
 \newtheorem{rem}[thm]{Remark}
 
 \numberwithin{equation}{section}

\def\D{\,\mathrm d}
\def\e{\mathrm e}
\newcommand{\DS}{\displaystyle}

\def\D{\mathrm{d}}

\numberwithin{equation}{section}

\title{Frames for weighted shift-invariant spaces}
\author{Stevan Pilipovi\' c$^{1}$, Suzana
Simi\'c$^{2}$}

\begin{document}
\date{August 4, 2010}
\maketitle


\noindent
{\bf Abstract.}
In this paper we prove the equivalence of the frame property  and the closedness
for a weighted shift-invariant space
\[V^p_{\mu}(\Phi)=\Big\{\sum\limits_{i=1}^r\sum\limits_{j\in{\mathbb{Z}}^d}c_i(j)\phi_i(\cdot-j)
\;\Big|\;
\{c_i(j)\}_{j\in{\mathbb{Z}}^d}\in\ell^p_{\mu}\Big\},\quad
p\in[1,\infty],\] which  corresponds to
$\Phi=\Phi^r=(\phi_1,\phi_2,\ldots,\phi_r)^T\in(W^1_\omega)^r$. We, also, construct a sequence $\Phi^{2k+1}$ and the
sequence of spaces $V^p_\mu(\Phi^{2k+1})$, $k\in{\mathbb{N}}$, on
$\mathbb{R},$ with the useful properties in sampling,
approximations and stability.
\medskip


\medskip



\section{Introduction}\label{sec:1}

\indent

In this paper, we investigate weighted shift-invariant spaces quoted in the
abstract by following the methods from \cite{ast} and \cite{shin}.
 Such spaces fi\-gure in several areas of applied
mathematics, notably in wavelet theory and approximation theory
(\cite{ast}, \cite{bor}).
In recent years, they have been extensively studied  by many authors (see \cite{ac}-\cite{bor},
 \cite{fei7}-\cite{fei9}, \cite{ji3}, \cite{ji2}, \cite{shin}, \cite{xs}).
Sampling with non-bandlimited functions in shift-invariant spaces
is a suitable and realistic model for many applications, such as
modeling signals with the spectrum that is smoother then in the case of
bandlimited
 functions, or for the numerical implementation (see \cite{ni1}, \cite{deja}, \cite{fei2}, \cite{fei6}, \cite{fei3},
  \cite{fei10}).
 These requirements can often be met by choosing appropriate functions in $\Phi$.
 This means that the functions in  $\Phi$ have a shape corresponding to a parti\-cular
  impulse response of a device, or that they are compactly supported or that
  they have a Fourier transform  decaying  smoothly
  to zero as $|\xi|\rightarrow \infty$.

Weighted shift-invariant spaces $V^p_\mu(\Phi)$, $p\in[1,\infty]$, where
$\mu$ is a weight, were introduced for the non uniform sampling as
a direct generalization of the space $V^p(\Phi)$ (\cite{ac},
\cite{xs}). The determination of $p$ and the signal smoothness are
used for optimal compression and coding signals and images (see
\cite{deja}).

The first aim of this paper is to show that the main result of
\cite{ast} holds in the case of weighted shift-invariant spaces
which correspond
 to $L^p_{\mu}$ and $\ell^p_{\mu}$, i.e., weighted $L^p$ and $\ell^p$ spaces,
respectively. Namely, we follow \cite{ast} and \cite{shin} and prove assertions which
need additional arguments depending on the weights. We show that
under the appropriate conditions on the frame vectors, there is an
equivalence between the concept of $p$-frames, Banach frames with
respect to $\ell^p_\mu$ and closedness of the space which they
generate. A weighted analog of Corollary 3.2 from \cite{shin} simplifies a part of the proof  of our main result. Although another part of the proof follows, step by step, the proof of the
corresponding theorem in \cite{ast}, we think that it is not simple
at all, and that it is worth to be done.

The second aim of this paper is to construct
$V^p_{\mu}(\Phi^{2k+1})$ spaces with specially
chosen functions, $\phi_0$, $\phi_1$, $\ldots$, $\phi_{2k}$, that generate a
Banach frame for the shift-invariant space $V^p_{\mu}(\Phi^{2k+1})$.
Actually, we take functions from a sequence
$\{\phi_i\}_{i\in{\mathbb{Z}}}$ so that the sequence of Fourier
transforms $\widehat{\phi}_i=\theta(\cdot+i\pi)$,
${i\in{\mathbb{Z}}}$, $\theta\in C^{\infty}_0(\mathbb{R})$, makes
a partition of unity in the frequency domain
($\mathbb{Z}=\mathbb{N}_0\cup-\mathbb{N}, \mathbb{N}$ is the set
of natural numbers and $\mathbb{N}_0=\mathbb{N}\cup\{0\}$). We
note that properties of the constructed
 frame guarantee the feasibility of a stable and continuous
 reconstruction algorithm in  $V^p_{\mu}(\Phi)$ (\cite{xs}).
 Also, we note that
 $\{\phi_i(\cdot-k)\mid k\in{\mathbb{Z}},\, i=1,\ldots,r\}$
 forms a Riesz basis for $V^p_{\mu}(\Phi)$ when the
 spectrum of the Gram matrix
 $[\widehat{\Phi},\widehat{\Phi}](\xi)$ is bounded and bounded
 away from zero (see \cite{bor}). The $d$-dimensional case, $d>1,$ is technically more complicated and because of that it is
not considered in this paper.

The paper is organized as follows. In Section \ref{sec:2} we quote
basic properties of subspaces of weighted $L^p$ and $\ell^p$
spaces. The weighted shift-invariant spaces are investigated  in
Section \ref{sec:3}, where we presented our first result quoted in
the abstract, Theorem~\ref{main}.  In Section \ref{sec:4} we show relations between the
dual of the Fr\' echet space
$\bigcap\limits_{s\in{\mathbb{N}}_0}V^p_{(1+|x|^2)^{s/2}}(\Phi)$
and the space of periodic distributions. The case of periodic
ultradistributions is obtain by using subexponential growth
functions. In Section \ref{sec:6}, we use a special sequence of functions $\{\phi_{k}\mid k\in{\mathbb{N}}\}$ to construct a sequence of
$p$-frames. Our construction shows that the sampling and
reconstruction problem in the shift-invariant spaces is robust.  In the final remark of
Section \ref{sec:6}, we
list good properties of these frames.

\section{Basic spaces}\label{sec:2}

\indent

Denote by $L^1_{loc}(\mathbb{R}^d)$ the space of measurable functions integrable over compact subsets of $\mathbb{R}^d$.  For a nonnegative function $\omega\in
L^1_{loc}({\mathbb{R}}^d)$ we say that is submultiplicative if
$\omega(x+y)\le\omega(x)\omega(y)$, $x$,
$y\in{\mathbb{R}}^{d}$ and a  function $\mu $ on
${\mathbb{R}}^{d}$ is $\omega$-moderate if $\mu (x+y)\le
C\omega(x)\mu (y)$, $x$, $y\in{\mathbb{R}}^{d}$. We assume
that $\omega$ is continuous and symmetric and both $\mu$ and
$\omega$  call  weights, as usual. The standard class of weights on
${\mathbb{R}}^{d}$ are of the polynomial type $\omega_s(x)=(1+|x|)^s$,
$s\ge 0$. To quantify faster decay of functions we use the
subexponential weights $\omega(x)=\e^{\alpha|x|^\beta}$, for some
$\alpha>0$ and $0<\beta<1$. Weighted $L^p$ spaces with moderate
weights are translation-invariant spaces (see \cite{ac}).
We, also,  consider weighted sequence spaces
$\ell^p_\mu({\mathbb{Z}}^d)$ with $\omega$-moderate weight  $\mu$.
Recall, a sequence $c$ belongs to $\ell^p_\mu({\mathbb{Z}}^d)$ if
$c\mu$ belongs to $\ell^p({\mathbb{Z}}^d)$.

In the sequel $\omega$ is a submultipicative weight, continuous
and symmetric and $\mu$ is $\omega$-moderate. Let
$p\in[1,\infty)$. Then (with obvious modification for $p=\infty$)
\[\mathcal{L}^p_\omega= \Big\{f\;\big|\;
\|f\|_{\mathcal{L}^p_\omega}=\Big(\int\limits_{[0,1]^d}\Big(\sum\limits_{j\in{\mathbb{Z}}^d}|f(x+j)|\omega(x+j)\Big)
^p\D x\Big)^{1/p}<+\infty\Big\},\]

\[{W}^p_\omega:=\Big\{f\;\big|\;\|f\|_{{W}^p_\omega}=\Big(\sum\limits_{j\in{\mathbb{Z}}^d}\sup
\limits_{x\in[0,1]^d}|f(x+j)|^p\omega(j)^p\Big)^{1/p}<+\infty\Big\}.\]
Obviously, we have $W^p_\omega\subset W^q_\omega
\subset\mathcal{L}^\infty_\omega\subset\mathcal{L}^q_\omega\subset\mathcal{L}^p_\omega\subset
L^p_\omega$, $W^p_\omega\subset W^p_\mu\subset W^q_\mu\subset
L^q_\mu\quad\mbox{and}\quad L^p_\omega\subset L^p_\mu$, where $1<
p< q\le+\infty$. For $p=1$ and $\omega=1$ we have
$\mathcal{L}^1=L^1$. We also have
$\ell^1_\omega\subset\ell^p_\omega\subset\ell^q_\omega\subset\ell^q_\mu$,
for $1< p< q\le+\infty$. From \cite{ac} we have the following properties.
\begin{itemize}
\item [$1)$] If $f\in L^p_\mu$, $g\in L^1_\omega$ and
$p\in[1,\infty]$, then $
\|f*g\|_{L^p_\mu}\le\|f\|_{L^p_\mu}\|g\|_{L^1_\omega}. $
\item [$2)$] If $f\in L^p_\mu$, $g\in W^1_\omega$ and
$p\in[1,\infty]$, then $
\|f*g\|_{W^p_\mu}\le\|f\|_{L^p_\mu}\|g\|_{W^1_\omega}. $
\item[$3)$] If $c\in\ell^p_\mu$ and $d\in\ell^1_\omega$, then
holds the inequality $ \|c*d\|_{\ell^p_\mu}\le
\|c\|_{\ell^p_\mu}\|d\|_{\ell^1_\omega}. $
\end{itemize}

Denote by $\mathcal{WC}^p_\mu$, $p\in[1,\infty]$, a space of all
$2\pi$-periodic functions with their sequences of Fourier
coefficients in $\ell^p_\mu$. Let $D_1$ and $D_2$ be the sequences
of Fourier coefficients of  $2\pi$-periodic functions $K_1$ and
$K_2$, respectively. If $D_1*D_2\in\ell^p_\mu$, then $D_1*D_2$ is the
sequence of Fourier coefficients of the product $K_1K_2$. For
$K=(K_1,\ldots,K_r)^T\in\mathcal({WC}^p_\mu)^r$, ($T$ means
transpose) define $\|K\|_{\ell^p_{\mu,*}}$ to be the $\ell^p_\mu$
norm of its sequence of Fourier coefficients.

In the sequel we use the notation
$\Phi=(\phi_1,\ldots,\phi_r)^T$. Define
$\|\Phi\|_{\mathcal{H}}=\sum\limits_{i=1}^r\|\phi_i\|_{\mathcal{H}}$,
where $\mathcal{H}=L^p_\omega$, $\mathcal{L}^p_\omega$ or
$W^p_\omega,\; p\in[1,\infty]$.

We list several lemmas needed to prove our results. Their proofs are
analogous to the proof of the corresponding lemmas in \cite{ast}.
\begin{lem}\label{lema2}
Let $f\in L^p_\mu$ and $g\in W^1_\omega$, $p\in[1,\infty]$. Then
the sequence
\[\Big\{\int\limits_{{\mathbb{R}}^d}f(x)g(x-j)\D x\Big\}_{j\in{\mathbb{Z}}^d}\in\ell^p_\mu\]
and $
\Big\|\Big\{\int\limits_{{\mathbb{R}}^d}f(x)g(x-j)\D x\Big\}_{j\in{\mathbb{Z}}^d}\Big\|_{\ell^p_\mu}\le
\|f\|_{L^p_\mu}\|g\|_{W^1_\omega}.$
\end{lem}

Let $c=\{c_i\}_{i\in{\mathbb{N}}}\in\ell^p_\mu$ and $f\in
L^p_\omega$, $p\in[1,\infty]$. We define, as in \cite{ast}, their
semi-convolution $f*'c$ by
$(f*'c)(x)=\sum\limits_{j\in{\mathbb{Z}}^d}c_jf(x-j),\;
x\in{\mathbb{R}}^d.$

\begin{lem}\label{lema3}\begin{itemize}
\item[$a)$]\label{semikon1} If $f\in L^p_\omega$ and
$c\in\ell^p_\mu$,$\;$$p\in[1,\infty]$, then $f*'c\in L^p_\mu$ and
$ \|f*'c\|_{L^p_\mu}\le
\|c\|_{\ell^p_\mu}\|f\|_{L^p_\omega}. $ \item[$b)$] If
$f\in\mathcal{L}^p_\omega$, $p\in[1,\infty]$, and
$c\in\ell^1_\mu$, then $ \|f*'c\|_{\mathcal{L}^p_\mu}\le
\|c\|_{\ell^1_\mu}\|f\|_{\mathcal{L}^p_\omega}. $ \item[$c)$]
\label{semikon3} If $f\in W^p_\omega$, $p\in[1,\infty]$, and
$c\in\ell^1_\mu$, then $
\|f*'c\|_{W^p_\mu}\le\|c\|_{\ell^1_\mu}\|f\|_{ W^p_\omega},
$ \item[$d)$]\label{semikon5} If  $f\in W^1_\omega$ and
$c\in\ell^p_\mu$, $p\in[1,\infty]$, then $
\|f*'c\|_{W^p_\mu}\le\|c\|_{\ell^p_\mu}\|f\|_{ W^1_\omega}.
$
\end{itemize}
\end{lem}

\section{Characterization of $V_\mu^p({\Phi})$}\label{sec:3}

\indent

In \cite{feigro} Feichtinger and Gr\"ochening  extended the notation
of atomic decomposition
 to Banach spaces (\cite{fei2}, \cite{fei6}), while Gr\"ochening \cite{groc}
 introduced a more general concept of decomposition through Banach
 frames. We recall the definition.

Let $X$ be a Banach space and $\Theta$ be an associated Banach
space of scalar valued sequences, indexed by $I=\mathbb{N}$ or
$I=\mathbb{Z}$. Let $\{f_n\}\subset X^*$ and $S:\Theta\rightarrow
X$ be given. The pair $(\{f_n\}_{n\in I},S)$ is called a Banach
frame for $E$ with respect to $\Theta$ if
\begin{itemize}
\item[$(1)$] \ $\{f_n(x)\}_{n\in I}\in \Theta$ for each $x\in X$,
\item[$(2)$] \ there exist positive constants $A$ and $B$ with
$0<A\le B<+\infty$ such that
$A\|x\|_X\le\|\{f_n(x)_{n\in I}\}\|_\theta\le
B\|x\|_X$, $x\in X$,
    \item[$(3)$] \ $S$ is a bounded linear operator such that $S(\{f_n(x)\}_{n\in I})=x$, $x\in X$.
    \end{itemize}
\noindent It is said that a collection $\{\phi_i(\cdot-j)\mid j\in
{\mathbb{Z}}^d, 1\le i\le r\}$ is a $p$-frame for
$V^p_\mu(\Phi)$ if there exists a positive constant $C$ (depending
on $\Phi$, $p$ and $\omega$)
\begin{equation}\label{pframe}
C^{-1}\|f\|_{L^p_\mu}\le\sum\limits_{i=1}^r\Big\|\Big\{\int\limits_{{\mathbb{R}}^d}f(x)
\phi_i(x-j)\D x\Big\}_{j\in{\mathbb{Z}}^d}\Big\|_{\ell^p_\mu}\le
C\|f\|_{L^p_\mu},\quad f\in V^p_\mu(\Phi).
\end{equation}

A typical application is the problem of finding a shift-invariant
space model that describes a given class of signals or images
(e.g. the class of chest $X$-rays). The observation set
of $r$ signals or images $f_1,\ldots,f_r$ may be theoretical
samples, or experimental data.

Recall \cite{ac}, the shift-invariant spaces are defined by
\[V^p_\mu(\Phi):=\Big\{f\in L^p_\mu\;\mid\;
f(\cdot)=\sum\limits_{i=1}^r\sum\limits_{j\in{\mathbb{Z}}^d}{c^i_j}\,\phi_i(\cdot-j),
\;\;\{c^i_j\}_{j\in{\mathbb{Z}}^d}\in\ell^p_\mu,\;1\le
i\le r\Big\}.\]

\begin{rem}
If $\Phi\in
 W^1_\omega$ and $\mu$ is $\omega$-moderate, then $V^p_\mu(\Phi)$ is a subspace (not necessarily closed) of
$L^p_\mu$ and $W^p_\mu$ for any $p\in[1,\infty]$. If $r=1$ and
$\{\phi(\cdot-j)\mid j\in{\mathbb{Z}}^d\}$ is a $p$-frame for
$V^p_\mu(\phi)$, then $V^p_\mu(\phi)$ is a closed subspace of
$L^p_\mu$ and $W^p_\mu$ for $p\in[1,\infty]$ (see \cite{suza}).
\end{rem}
Let $[\widehat{\Phi},\widehat{\Phi}](\xi)=
\Big[\sum\limits_{k\in{\mathbb{Z}}^d}\widehat{\phi}_i(\xi+2k\pi)\overline{\widehat{\phi}_j(\xi+2k\pi)}\,\Big]_{1\le
i\le r,\;1\le j\le r}$ where we assume that
$\widehat{\phi_i}(\xi)\overline{\widehat{\phi}_j(\xi)}$ is
integrable for any $1\le i,j\le r.$ Let
$A=(a(j))_{j\in{\mathbb{Z}}^d}$ be an $r\times \infty$ matrix and
$A\overline{A^T}=\Big[\sum\limits_{j\in{\mathbb{Z}}^d}a_i(j)\overline{a_{i'}(j)}\Big]_{1\le
i,i'\le r}$. Then $ \rank A=\rank A\overline{A^T} $.

Also, since $[\widehat{\Phi},\widehat{\Phi}](\xi)$ is continuous
(as a function with $r^2$ components) for any
$\Phi\in(\mathcal{L}^2_\omega)^r$, it follows that
$\big\{\xi\in{\mathbb{R}}^d\mid
\rank\big[\widehat{\Phi}(\xi+2k\pi)_{k\in{\mathbb{Z}}^d}\big]>k_0\big\}$
is an open set for any $k_0>0$ and
$\Phi\in(\mathcal{L}^2_\omega)^r$.

Denote by $\Sigma_\alpha^\mu$ the family of all $\alpha$-slant matrices $A=[a(j,k)_{j\in\mathbb{Z}^d,k\in\mathbb{Z}^d}]$ with
\[\|A\|_{\Sigma_\alpha^\omega}=\sum\limits_{k\in\mathbb{Z}^d}\sup\limits_{j,k\in\mathbb{Z}^d}|a(k,j)|\chi_{k+[0,1)^d}(k-\alpha j)<\infty,\] where $\mu$ is a weight on $\mathbb{R}^d$ and $\alpha$ is a positive number.
The slanted matrices appear in wavelet theory, signal processing and sampling theory (see \cite{shin}). Note $\Sigma^\mu_\alpha\subset\Sigma_\alpha^{\mu_0}$ for any weight $\mu$ where $\mu_0\equiv1$ is the trivial weight.

We assume in this subsection that
$\Phi=(\phi_1,\ldots,\phi_r)^T\in(\mathcal{L}^p_\omega)^r$ for
$p\in[1,\infty)$.

To prove Theorem~\ref{main}, we need several lemmas. First we recall a  result from \cite{ast}.
\begin{lem}[\cite{ast}]\label{lema1ald}
The following statements are equivalent.
\begin{itemize}
\item[$1)$]
$\rank\big[\widehat{\Phi}(\xi+2j\pi)_{j\in{\mathbb{Z}}^d}\big]$
 is a constant function on ${\mathbb{R}}^d$.
\item[$2)$] $\rank[\widehat{\Phi},\widehat{\Phi}](\xi)$ is a
constant function on ${\mathbb{R}}^d$. \item[$3)$] There exists a
positive constant $C$ independent of $\xi$ such that
\[C^{-1}[\widehat{\Phi},\widehat{\Phi}](\xi)\le [\widehat{\Phi},\widehat{\Phi}]
(\xi)\,\overline{[\widehat{\Phi},\widehat{\Phi}](\xi)^T}\le
C\,[\widehat{\Phi}, \widehat{\Phi}](\xi), \quad
\xi\in[-\pi,\pi]^d.\]
\end{itemize}
\end{lem}
The proofs of the following two lemmas  are similar to proofs of the corresponding lemmas from \cite{ast}; hence we will not include them here.
The second one provides a localization technique in Fourier
domain. It allows us to replace locally the generator
$\widehat{\Phi}$ of size $r$ by $\widehat{\Psi}_{1,\lambda}$ of
size $k_0$.

\begin{lem} \label{ubac} All the entries of
$[\widehat{\Phi},\widehat{\Phi}](\xi)$ belong to
$\mathcal{WC}^1_\omega$ and are continuous.
\end{lem}

\begin{lem}\label{lema2ald}
Let the
$\rank[\widehat{\Phi}(\xi+2j\pi)_{j\in{\mathbb{Z}}^d}]=k_0\ge
1$ for all $\xi\in{\mathbb{R}}^d$. Then there exist a finite index
set $\Lambda$, points $\eta_\lambda\in[-\pi,\pi]^d$,
$0\le\delta_\lambda<1/4$, a nonsingular $2\pi$-periodic
$r\times r$ matrix $P_\lambda(\xi)$ with all entries in the class
$\mathcal{WC}^1_\omega$ and $K_\lambda\subset {\mathbb{Z}}^d$ with
cardinality  $k_0$ for all $\lambda\in\Lambda$, such that:

$(i)$ $
\bigcup\limits_{\lambda\in\Lambda}B(\eta_\lambda,\delta_\lambda/2)\supset[-\pi,\pi]^d
$, where $B(x_0,\delta)$ denotes the open ball in ${\mathbb{R}}^d$
with center $x_0$ and radius $\delta$;

$(ii)$ $P_\lambda(\xi)\widehat{\Phi}(\xi)=\left[\begin{array}{ll}
\widehat{\Psi}_{1,\lambda}(\xi) \\
\widehat{\Psi}_{2,\lambda}(\xi)
\end{array}\right]$,
$\xi\in{\mathbb{R}}^d$, $\lambda\in\Lambda $, where
$\Psi_{1,\lambda}$ and $\Psi_{2,\lambda}$ are functions from
${\mathbb{R}}^d$ to $C^{k_0}$ and $C^{r-k_0}$, respectively,
satisfying
\[
\rank\big[\widehat{\Psi}_{1,\lambda}(\xi+2k\pi)_{k\in
K_\lambda}\big]=k_0,\quad \xi\in B(\eta_\lambda,2\delta_\lambda),
\]
\[
\widehat{\Psi}_{2,\lambda}(\xi)=0,\quad \xi\in
B(\eta_\lambda,8\delta_\lambda/5)+2\pi{\mathbb{Z}}^d.\]
Furthermore, there exist $2\pi$-periodic $C^\infty$ functions
$h_\lambda$, $\lambda\in\Lambda$, on ${\mathbb{R}}^d$ such that $
\sum\limits_{\lambda\in\Lambda}h_\lambda(\xi)=1$,
$\xi\in{\mathbb{R}}^d$, and $ \mbox{supp}\,h_\lambda\subset
B(\eta_\lambda,\delta_\lambda/2)+2\pi{\mathbb{Z}}^d. $
\end{lem}

The next lemma is needed for the proof of Theorem \ref{main}. Although the formulation is not the same as \cite[Lemma 3]{ast},  the proof
 is based on the same procedure, and we  omit it.

\begin{lem}\label{lema3zas}
$(a)$ \ Let $\phi\in\mathcal{L}^p_{\omega_s}$ if $p\in[1,\infty)$
and $\phi\in W^1_{\omega_s}$ if $p=+\infty$. Assume that
$\sum\limits_{j\in{\mathbb{Z}}^d}\phi(\cdot+j)=0$. Then for any
function $h$ on ${\mathbb{R}}^d$ which satisfies
\[
|h(x)|\le C(1+|x|)^{-s-d-1} ,\quad |h(x)-h(y)|\le
C\frac{|x-y|}{(1+\min\{|x|,|y|\})^{s+d+1}},
\]
we have
\[
\lim\limits_{n\rightarrow+\infty}2^{-nd}\Big\|\sum\limits_{j\in{\mathbb{Z}}^d}h(2^{-n}j)
\phi(\cdot-j)\Big\|_{\mathcal{L}^p_{\omega_s}}=0.
\]

$(b)$ \ Let $\mu(x)=\e^{\alpha|x|^\beta}$. Let
$\phi\in\mathcal{L}^p_{\mu}$ if $p\in[1,\infty)$ and $\phi\in
W^1_{\mu}$ if $p=+\infty$. Assume that
$\sum\limits_{j\in{\mathbb{Z}}^d}\phi(\cdot+j)=0$. Then for any
function $h$ on ${\mathbb{R}}^d$ which satisfies
\[
|h(x)|\le C\e^{-(\alpha+d+1)|x|^\beta},\quad
|h(x)-h(y)|\le
C|x-y|\e^{-(\alpha+d+1)(1+\min\{|x|^\beta,|y|^\beta\})},
\]
we have
\[
\lim\limits_{n\rightarrow+\infty}2^{-nd}\Big\|\sum\limits_{j\in{\mathbb{Z}}^d}h(2^{-n}j)
\phi(\cdot-j)\Big\|_{\mathcal{L}^p_{\mu}}=0.
\]
\end{lem}

Next, we give a result  on the equivalence of $\ell^p_\mu$-stability of the synthesis operator $S_\Phi$ for a different $p\in[1,\infty]$ (see \cite{shin}; here we have  $\Lambda=\{1,2,\ldots,r\}$).

\begin{prop}\label{pos:sta1}\cite[Corollary 3.2]{shin}
Let $\Phi=(\phi_1,\ldots,\phi_r)^T\in(W^1_\omega)^r$, $p\in[1,\infty]$ and $\mu$ is $\omega$-moderate. Define the synthesis operator $S_{\Phi}:(\ell^p_\mu(\mathbb{Z}^d))^r\mapsto V_\mu^p(\Phi)$ by
\begin{equation*}\label{jed:stabilnostzap}S_{\Phi}:c=\{c^i_j\}_{j\in\mathbb{Z}^d,1\le i\le r}\mapsto\sum\limits_{i=1}^r\sum\limits_{j\in\mathbb{Z}^d}c^i_j\phi_i(\cdot-j).\end{equation*}
If the synthesis operator $S_{\Phi}$ has $\ell^p_\mu$-stability for some $p\in[1,\infty]$, i.e., there exists a positive constant $C$ such that
\begin{equation}\label{pro:zariesz}C^{-1}\|c\|_{(\ell^p_\mu(\mathbb{Z}^d))^r}\leq \|S_{\Phi} c\|_{L^p_\mu}\leq C\|c\|_{(\ell^p_\mu(\mathbb{Z}^d))^r},\end{equation}
for all $c\in {(\ell^p_\mu(\mathbb{Z}^d))^r}$, then the synthesis operator $S_{\Phi}$ has $\ell^q_\mu$-stability for any $q\in[1,\infty]$.
\end{prop}

As a consequence of the previous proposition, we have the next result.

\begin{prop}\label{pos:sta2}\cite[Corollary 3.3]{shin}
Let $p\in[1,\infty]$ and $\Phi=(\phi_1,\ldots,\phi_r)^T\in (W^1_\omega)^r$, and $\mu$ $\omega$-moderate. If the synthesis operator $S_\Phi$ has $\ell^p_\mu$-stability, then there exists another family $\Psi=(\psi_1,\ldots,\psi_r)^T\in(W^1_\omega)^r$ such that the inverse of the synthesis operator $S_\Phi$ is given by
\[(S_\Phi)^{-1}(f)=\Bigl\{\int\limits_{\mathbb{R}^d}f(x)\psi_i(x-j)\D x\Bigr\}_{1\le i\le r,j\in\mathbb{Z}^d},\; f\in V^p_\mu.\]
\end{prop}
Proposition~\ref{pos:sta1} and  \ref{pos:sta2} imply:
\begin{thm}\label{mala-main}
Let $\Phi=(\phi_1,\ldots,\phi_r)^T\in (W^1_\omega)^r$,
$p_0\in[1,\infty]$, and $\mu$ is $\omega$-moderate. Then the
following three statements are equivalent.
\begin{itemize}
\item[$a)$] The synthesis operator $S_{\Phi}$ has $\ell^{p_0}_\mu$-stability.
\item[$b)$] \ $V^{p_0}_\mu(\Phi)$ is closed in $L^{p_0}_\mu$.
\item[$c)$] \ There exists $\Psi=(\psi_1,\ldots,\psi_r)^T\in
(W^1_\omega)^r$, such that
\[f=\sum\limits_{i=1}^r\sum\limits_{j\in{\mathbb{Z}}^d}\langle
f,\psi_i(\cdot-j)\rangle\phi_i(\cdot-j),\;\;f\in
V^{p_0}_\mu(\Phi).\]
\end{itemize}
Also we have the next assertion.

$d)$ If the synthesis operator $S_{\Phi}$ has $\ell^{p_0}_\mu$-stability, then the collection $\{\phi_i(\cdot-j)\mid
j\in{\mathbb{Z}}^d,1\le i\le r\}$ is a $p_0$-frame for
$V^{p_0}_\mu(\Phi)$.
\end{thm}
\begin{proof}
The implication $a)\Rightarrow c)$ is a consequence of Proposition~\ref{pos:sta1} (see Proposition~\ref{pos:sta2}).

$c)\Rightarrow a)$: Let $f=\sum\limits_{i=1}^r\sum\limits_{j\in{\mathbb{Z}}^d}\langle
f,\psi_i(\cdot-j)\rangle\phi_i(\cdot-j)$ and  \[c^i=\{\langle f,\psi_i(\cdot-j) \rangle\}_{j\in\mathbb{Z}^d},\; \;1\le i\le r.\] Then
\[\|c\|_{(\ell^p_\mu)^r}=\sum\limits_{i=1}^r\Big\|\Big\{\int\limits_{{\mathbb{R}}^d}f(x)
\psi_i(x-j)\D x\Big\}_{j\in{\mathbb{Z}}^d}\le C
\|f\|_{L^{p_0}_\mu},\] where $C=\sum\limits_{i=1}^r\|\psi_i\|_{W^1_\omega}$. Using Lemma~\ref{lema2} and the inequality (\ref{lema3}), we obtain the right-hand side of (\ref{pro:zariesz}).

The equivalence $a) \Leftrightarrow b)$ follows from standard functional analytic arguments (see \cite[Theorem 2, Lemma 4]{ast}).

$d)$ Lemma~\ref{lema2} implies that $\{\langle f,\phi_i(\cdot-j)\rangle\}\in\ell^{p_0}_\mu$, $1\le i\le r$, and
\[
\sum\limits_{i=1}^r\Big\|\Big\{\int\limits_{{\mathbb{R}}^d}f(x)
\phi_i(x-j)\D x\Big\}_{j\in{\mathbb{Z}}^d}\le
\|f\|_{L^{p_0}_\mu}\sum\limits_{i=1}^r\|\phi_i\|_{W^1_\omega}.\] Now, $\ell^p_\mu$-stability implies
\[\|f\|_{L^{p_0}_\mu}\le C\sum\limits_{i=1}^r\Big\|\Big\{\int\limits_{{\mathbb{R}}^d}f(x)
\phi_i(x-j)\D x\Big\}_{j\in{\mathbb{Z}}^d}\Big\|_{\ell^{p_0}_\mu}.\]
\end{proof}

\begin{rem}
Note that $\ell^p_\mu$-stability of the synthesis operator implies $\ell^q_\mu$-stability, for any $q\in[1,\infty]$ (\cite{shin}), so the statements $b)$, $c)$ and $d)$, do not depend on $p\in[1,\infty]$.
\end{rem}
Now, we give our main result.

\begin{thm}\label{main}
Let $\Phi=(\phi_1,\ldots,\phi_r)^T\in (W^1_\omega)^r$,
$p_0\in[1,\infty]$, and $\mu$ is $\omega$-moderate. Then the
following statements are equivalent.
\begin{itemize}
\item[$i)$] \ $V^{p_0}_\mu(\Phi)$ is closed in $L^{p_0}_\mu$.
\item[$ii)$]\label{dva} \ $\{\phi_i(\cdot-j)\mid
j\in{\mathbb{Z}}^d,1\le i\le r\}$ is a $p_0$-frame for
$V^{p_0}_\mu(\Phi)$. \item[$iii)$] \ There exists a positive
constant $C$ such that
\[C^{-1}[\widehat{\Phi},\widehat{\Phi}](\xi)\le[\widehat{\Phi},
\widehat{\Phi}](\xi)\overline{[\widehat{\Phi},\widehat{\Phi}](\xi)^T}\le
C[\widehat{\Phi},\widehat{\Phi}](\xi),\quad \xi\in[-\pi,\pi]^d.\]
\item[$iv)$] There exist positive constants $C_1$ and $C_2$
(depending on $\Phi$ and $\omega$) such that
\begin{equation}\label{cetiri}
C_1\|f\|_{L^{p_0}_\mu}\le\inf\limits_{ f
=\sum\limits_{i=1}^r\phi_i*'c^i}\sum\limits_{i=1}^r\|\{c^i_j\}_{j\in{\mathbb{Z}}^d}\|_{\ell^{p_0}_\mu}\le
C_2\|f\|_{L^{p_0}_\mu},\;\;f\in V^{p_0}_\mu(\Phi).
\end{equation}
\item[$v)$] \ There exists $\Psi=(\psi_1,\ldots,\psi_r)^T\in
(W^1_\omega)^r$, such that
\[f=\sum\limits_{i=1}^r\sum\limits_{j\in{\mathbb{Z}}^d}\langle
f,\psi_i(\cdot-j)\rangle\phi_i(\cdot-j)=\sum\limits_{i=1}^r\sum\limits_{j\in{\mathbb{Z}}^d}
\langle f,\phi_i(\cdot-j)\rangle\psi_i(\cdot-j),\;\;f\in
V^{p_0}_\mu(\Phi).\]
\end{itemize}
\end{thm}
\begin{proof}
If the synthesis operator  has $\ell^p_\mu$-stability, then the statement $iv)$ is satisfied. Conversely, if the statement $iv)$ is satisfied, then the right-hand side of (\ref{pro:zariesz}) (with $p=p_0$) immediately follows. Using $c)$ from Theorem~\ref{mala-main}, we obtain the left-hand side of (\ref{pro:zariesz}). Hence, by Theorem \ref{mala-main}, we  have $i)\Leftrightarrow iv)$ and $iv)\Rightarrow  ii)$. The equivalence $iv)\Leftrightarrow v)$ follows from  Lemma~\ref{lema2}.

We follow \cite{ast} to prove $iii)\Rightarrow v)$ and $ii)\Rightarrow iii)$, and carefully check the use of weights.

$\bf{\it{iii)}\Rightarrow
 v)}$. Let $B_\lambda(\xi)=H_\lambda(\xi)\overline{P_\lambda(\xi)^T}\left(\begin{array}{ll}
\Big[\widehat{\Psi}_{1,\lambda},\widehat{\Psi}_{1,\lambda}\Big](\xi)^{-1}& \quad  0\\
\qquad\;\;0&  \quad  \mbox{I}
\end{array}\right)P_\lambda(\xi)$,  for
$h_\lambda(\xi)$, $P_\lambda(\xi)$ and
$\widehat{\Psi}_{1,\lambda}$ as in Lemma \ref{lema2ald}. We have
$B_\lambda(\xi)\in\mathcal{WC}^p_\omega$, for all $p\in[1,+\infty]$.
 Define
$\widehat{\Psi}(\xi)=\sum\limits_{\lambda\in\Lambda}h_\lambda(\xi)B_\lambda(\xi)\widehat{\Phi}(\xi)$.
One has $\Psi\in W^1_\omega$. For any $f\in V^p_\mu(\Phi)$, define
$g(x)=\sum\limits_{i=1}^r\sum\limits_{j\in{\mathbb{Z}}^d}\langle
f,\psi_i(x-j)\rangle\phi_i(x-j)$, $x\in{\mathbb{R}}^d$. Since
$f\in V^p_\mu(\Phi)$, there exists a $2\pi$-periodic distribution
$A(\xi)\in\mathcal{WC}^p_\mu$ such that
$\widehat{f}(\xi)=A(\xi)^T\widehat{\Phi}(\xi)$. By Lemma
\ref{lema2ald}, we have $\widehat{g}(\xi)=\widehat{f}(\xi)$.

Since $\widehat{\Psi}(\xi)=\sum\limits_{\lambda\in\Lambda}h_\lambda(\xi)B_\lambda(\xi)\widehat{\Phi}(\xi)$, for $f=\sum\limits_{i=1}^r\sum\limits_{j\in\mathbb{Z}^d}\langle f,\phi_i(\cdot-j)\rangle\psi(\cdot-j)$  the proof is similar.

 $\bf{\it{ii)}\Rightarrow
 iii)}$. Let $k_0=\min\limits_{\xi\in{\mathbb{R}}^d}\rank\big[\widehat{\Phi}(\xi+2k\pi)_{k\in{\mathbb{Z}}^d}\big]$
and let
\[\Omega_{k_0}=\big\{\xi\in{\mathbb{R}}^d\mid\rank\big[\widehat{\Phi}(\xi+2k\pi)_{k\in{\mathbb{Z}}^d}\big]>k_0\big\}.\]
Then $\Omega_{k_0}\neq{\mathbb{R}}^d$. It is sufficient to prove
that $\Omega_{k_0}=\emptyset$ (see Lemma \ref{lema1ald}).  Suppose that
$\Omega_{k_0}\neq\emptyset$. Since $\Omega_{k_0}$ is open set, then
$\partial\Omega_{k_0}\neq\emptyset$ and
$\rank\big[\widehat{\Phi}(\xi_0+2k\pi)\big]_{k\in{\mathbb{Z}}^d}=k_0$,
for any $\xi_0\in\partial\Omega_{k_0}$, and $\max\limits_{\xi\in
B(\xi_0,\delta)}\rank\big[\widehat{\Phi}(\xi+2k\pi)\big]_{k\in{\mathbb{Z}}^d}>k_0$,
$\delta>0.$ By Lemma \ref{lema2ald}, there exist a nonsingular
$2\pi$-periodic $r\times r$
 matrix $P_{\xi_0}(\xi)$
with all entries in the class $\mathcal{WC}^1_\omega$, $\delta_{0}>0$
and $K_{0}\subset {\mathbb{Z}}^d$ with cardinality $k_0$. Define
$\Psi_{\xi_0}$, $\widehat{\Psi}_{\xi_0}(\xi)$ as in Lemma
\ref{lema2ald}. The construction of
$\Psi_{\xi_0}$ and (\ref{semikon3})  imply $\Psi_{\xi_0}\in
W^1_\omega$.
 Choose $n_0$ such that
$2^{-n_0}<\delta_0$ and define $\alpha_n(\xi)$, $H_{n,\xi_0}(\xi)$ and $\widetilde{H}_{n,\xi_0}(\xi)$ as in \cite{ast}. For any
$2\pi$-periodic distribution $F\in\mathcal{WC}^{p_0}_\mu$ define, $g_n$, for $n\ge n_0+1$, as in \cite{ast}. Note that $g_n\in
V^{p_0}_\mu(\Phi)$ and
$[\widehat{g}_n,\widehat{\Psi}_{1,\xi_0}](\xi)=0$. This leads to
\begin{align*}
\|[\widehat{g}_n,\widehat{\Phi}](\xi)\|_{\ell^{p_0}_{\mu,*}}&\le
C\|g_n\|_{L^{p_0}_\mu}
\|\mathcal{F}^{-1}(H_{n,\xi_0}(\xi)\widehat{\Psi}_{2,\xi_0}(\xi))\|_{\mathcal{L}^\infty_\omega}.
\end{align*}
Using Lemma~\ref{lema3zas}, we obtain
$\lim\limits_{n\rightarrow+\infty}\|\mathcal{F}^{-1}(H_{n,\xi_0}(\xi)\widehat{\Psi}_{2,\xi0}(\xi))
\|_{\mathcal{L}^\infty_\omega}=0$. There exists a
sequence $\rho_n$, $n\ge n_0$, such that
$\|[\widehat{g}_n,\widehat{\Phi}](\xi)\|_{\ell^{p_0}_{\mu,*}}\le
\rho_n\|g_n\|_{L^{p_0}_\mu}$ and $\lim\limits_{n\rightarrow
+\infty}\rho_n=0$. This, together with the assumption {\it{ii}}$)$
and
\[\|[\widehat{g}_n,\widehat{\Phi}](\xi)\|_{\ell^{p_0}_{\mu,*}}=\Big\|\Big\{
\int\limits_{{\mathbb{R}}^d}g_n(\xi)\overline{\Phi(\xi-j)}dx\Big\}_{j\in{\mathbb{Z}}^d}\Big\|_{\ell^{p_0}_\mu}
\ge C\|g_n\|_{L^{p_0}_\mu},\] leads to $g_n=0$, $n\ge
n_0+1$. Then
\begin{equation}\label{5.21}
\widetilde{H}_{n,\xi_0}(\xi)[\widehat{\Psi}_{1,\xi_0},\widehat{\Psi}_{1,\xi_0}](\xi)
(\alpha_n(\xi))^{-1}
\widehat{\Psi}_{1,\xi_0}(\xi)=\widetilde{H}_{n,\xi_0}(\xi)\widehat{\Psi}_{2,\xi_0}(\xi),
\end{equation}
for any $2\pi$-periodic distribution $F\in\mathcal{WC}^{p_0}_\mu$
and $n\ge n_0+1$. We, also, get
\[\widetilde{H}_{n,\xi_0}(\xi)[\widehat{\Psi}_{1,\xi_0},\widehat{\Psi}_{1,\xi_0}](\xi)
(\alpha_n(\xi))^{-1} \widehat{\Psi}_{1,\xi_0}(\xi)=0,\quad \xi\in
B(\xi_0,2^{-n_0-1})+2\pi{\mathbb{Z}}^d.\] So, from (\ref{5.21})
and the fact that it is valid for all $n\ge n_0+1$, we have
$\widehat{\Psi}_{2,\xi_0}(\xi)=0$, $\xi\in
B(\xi_0,2^{-n_0-3})+2\pi{\mathbb{Z}}^d$. This contradicts the fact
that $ \widehat{\Psi}_{2,\xi_0}(\xi)\neq0$, $\forall\xi\in
B(\xi_0,\delta)+2\pi{\mathbb{Z}}^d $, $0<\delta<2\delta_{0}$.

With this we complete the proof $\it{ii)\Rightarrow iii)}$ and the
proof of the theorem.

\end{proof}

\begin{rem}
Note that conditions in Theorem \ref{mala-main} and Theorem \ref{main} do not depend
on $p\in[1,\infty]$, so we obtain the next corollary.
\end{rem}
\begin{cor}\label{cor1}
Let $\Phi\in (W^1_\omega)^r$ and $p_0\in[1,\infty]$.
\begin{itemize}
\item[$i)$] \ If $\{\phi_i(\cdot-j)\mid
j\in{\mathbb{Z}}^d,1\le i\le r\}$ is a $p_0$-frame for
$V^{p_0}_\mu(\Phi)$, then $\{\phi_i(\cdot-j)\mid
j\in{\mathbb{Z}}^d,1\le i\le r\}$ is a $p$-frame for
$V^p_\mu(\Phi)$, for any $p\in[1,\infty]$. \item[$ii)$] \ If
$V^{p_0}_\mu(\Phi)$ is closed in $L^{p_0}_\mu$ and $W^{p_0}_\mu$,
then $V^p_\mu(\Phi)$ is closed in $L^p_\mu$ and $W^p_\mu$, for any
$p\in[1,\infty]$.
\end{itemize}
\end{cor}
\begin{rem}
$(v)\Rightarrow(ii)$ implies that  $\{\psi_i(\cdot-j)\mid
1\le i\le r,j\in\mathbb{Z}^d\}$
 is a dual $p$-frame of  $\{\phi_i(\cdot-j)\mid 1\le i\le r,j\in\mathbb{Z}^d\}$. So, the
 $p$-frame for $V^{p}_\mu(\Phi)$ is a Banach frame (with respect to
 $\ell^{p}_\mu$).
\end{rem}

\bigskip
\section{Connections with periodic distributions}\label{sec:4}

\indent

We will use the notation $V^p_s$ instead of
$V^p_{(1+|x|^2)^{s/2}}$ (similarly for $\ell^p_s$). Since
$\ell^p_s$ and $V^p_s$ are isomorphic Banach spaces for all
$s\ge 0$ and $p\in[1,\infty]$, we have
$V^p_{s_1}(\Phi)\subset V^p_{s_2}(\Phi)$ for $0\le
s_2\le s_1$, $p\in[1,\infty]$. We define Fr\' echet spaces
$X_{F,p}$, $p\in[1,\infty]$, as $\smash
X_{F,p}=\bigcap\limits_{s\in{\mathbb{N}}_0}V^p_s(\Phi)$. Clearly,
$X_{F,p}$ is dense in $V^p_s(\Phi)$ for all $s\in{\mathbb{N}}_0$.
The corresponding sequence space is
$Q_{F,p}=\bigcap\limits_{s\in{\mathbb{N}}_0}\ell^p_s$,
 $p\in[1,\infty]$, which is the space of rapidly decreasing sequences $s$. By Corollary
\ref{cor1} it follows that the definition of $X_{F,p}$ does not
depend on $p\in[1,\infty]$. So we use  notation $X_F$, $Q_F$
instead of $X_{F,p}$, $Q_{F,p}$. The set $\{\Phi(\cdot-k)\mid
k\in{\mathbb{Z}}^d\}$ forms a $F$-frame for $X_{F}$ since it forms
a Banach frame for every space in the intersection (see
\cite{suza} for the definition).

 Since the corresponding function space for $s$ is the
space of rapidly decreasing functions $\mathcal{S}=\{f\mid
\|f\|_m=\sup\limits_{n\le
m}(1+|x|^2)^{m/2}|f^{(n)}(x)|<+\infty\}$, and its dual is
$\mathcal{S}'$- the space of tempered  distributions, we obtain
that the dual space $X_F'$ is isomorphic to (a complemented
subspace of) $\mathcal{S}'$.

Denote by $\mathcal{P}(-\pi,\pi)$ the space of smooth $2\pi$-
periodic functions on $\mathbb{R}^d$ with the family of norms
$|\theta|_k=\sup\{|\theta^{(k)}(t)|\,; t\in(-\pi,\pi)\}$,
$k\in{\mathbb{N}}_0$. It is a Fr\' echet space and its dual is the
space of $2\pi$-periodic tempered distributions. We say that $T$
is a $2\pi$-periodic distribution if it is a tempered distribution
on $\mathbb{R}^d$ and $T=T(\cdot+2j\pi)$, for all
$j\in\mathbb{Z}^d$. Denote by $\mathcal{P}'(-\pi,\pi)$ the space
of periodic tempered distributions (see \cite{rss}).  Recall that
$\mathcal F(h)=\hat{h}=\int_{{\mathbb{R}}^d}\e^{-2\pi{\sqrt{-1}}
t\cdot}h(t)\D t$ for $h\in L^1.$
\begin{thm} Let
$\Phi=(\phi_1,\ldots,\phi_r)^T\in\bigcap\limits_{s\ge
0}(W^1_s)^r$ and $\Psi=(\psi_1,\ldots,\psi_r)^T$ be its dual frame
$($according to $v)$ of Theorem \ref{main}$)$.
 Then
\[X_F={\mathcal F}^{-1}\Big(\sum_{i=1}^r\hat{\phi_i}\cdot\mathcal{P}(-\pi,\pi)\Big),\qquad
X_F'={\mathcal
F}^{-1}\Big(\sum_{i=1}^r\hat{\psi_i}\cdot\mathcal{P}'(-\pi,\pi)\Big)
\]
in the topological sense. Let
\[f=\sum\limits_{k=1}^r\sum\limits_{p\in{\mathbb{Z}}^d}c_p^k\phi_k(\cdot-p)\in
X_F\;\;\mbox{and}\;\;F=\sum
\limits_{i=1}^r\sum\limits_{j\in{\mathbb{Z}}^d}d_j^i\psi_i(\cdot-j)\in
X_F'.\] The dual pairing is given by
\begin{equation}\label{dualpar}\langle F , f \rangle =\sum\limits_{i=1}^r\sum\limits_{k=1}^r\Big\langle \widehat{\psi}_i(\xi)\widehat{\phi}_k(-\xi)\sum
\limits_{j\in{\mathbb{Z}}^d}d^i_j\e^{2\pi
j\xi\sqrt{-1}},\sum\limits_{p\in{\mathbb{Z}}^d}c^k_p\e^{-2\pi
p\xi\sqrt{-1}} \Big\rangle,
\end{equation} where $f=\sum\limits_{k=1}^r\sum\limits_{p\in{\mathbb{Z}}^d}c_p^k\phi_k(\cdot-p)\in X_F$ and $F=\sum
\limits_{i=1}^r\sum\limits_{j\in{\mathbb{Z}}^d}d_j^i\psi_i(\cdot-j)\in
X_F'$.

\noindent In particular, we have
$\DS\int\limits_{{\mathbb{R}}^d}{\varphi_i}{\psi_k}\D t=\DS\int\limits_{{\mathbb{R}}^d}\widehat{\varphi_i}(\xi)\widehat{\psi_k}(-\xi)\D \xi=\delta_{ik}$,
$1\le i,k\le r$.
\end{thm}

\begin{proof} Since $\sum\limits_{p\in{\mathbb{Z}}^d}c^k_p\e^{2\pi{\sqrt{-1}}p\xi}\in\mathcal{P}(-\pi,\pi)$, we obtain the structure of $f\in X_F$ as in the theorem. The same explanation works for $X_F'$.

 By the fact that $\langle F(x),f(x)\rangle=\langle \widehat{F}(\xi),\widehat{f}(-\xi)\rangle$, we have that (\ref{dualpar}) fo\-llows.

Let $d^i_0=\delta_{ik}$, $i=1,\ldots,r$, and $d^i_j=0$, $j\neq 0$,
and, also, let $c^k_0=\delta_{ik}$ for $k=1,\ldots,r$ and
$c^k_p=0$, $p\neq0$. Using that, we obtain \[
\langle
F(\xi),f(\xi)\rangle=\sum\limits_{i=1}^r\sum\limits_{k=1}^r
 \langle \widehat{\psi}_i(\xi),\widehat{\phi}_k(-\xi)d_0^i,c_0^k
\rangle=
\int\limits_{{\mathbb{R}}^d}\widehat{\psi}_{k_0}(\xi)\widehat{\phi}_{k_0}(-\xi)\D\xi,\,1\le
k_0\le r.
\]

On the other hand $f(x)= \langle f(x),\psi_{k_0}(x)\rangle
\phi_{k_0}(x)$ and $f=\phi_{k_0}$ for some $1\le
k_0\le r$, so we obtain $\langle f,\psi_{k_0}\rangle=1$.
Since $F=\psi_{k_0}$, we get $\langle F,f\rangle= \langle
f,\psi_{k_0}\rangle=1$. Finally, we have
$\DS\int\limits_{{\mathbb{R}}^d}\widehat{\varphi_i}({\xi})\widehat{\psi_k}(-\xi)\D\xi=\delta_{ik}$,
$1\le i,k\le r$.
\end{proof}

Let $\beta\in(0,1)$. Now, we consider weights
$\mu_k=\e^{k|x|^\beta}$, $k\in{\mathbb{N}}$, and the co\-rresponding
spaces $V^p_{\mu_k}(\Phi)$ and their intersection
$X^{(\beta)}_{F,p}=\bigcap\limits_{k\in{\mathbb{N}}}V^p_{\mu_k}(\Phi)$.
It is a Fr\' echet space not depending on $p$, so we use notation
$X_F^{(\beta)}$. The corresponding sequence space is
$s^{(\beta)}=\bigcap\limits_{k\in{\mathbb{N}}}\ell^p_{\mu_k}$, i.e., the
space of subexponentially rapidly decreasing sequences determining
the space of periodic tempered ultradistributions via the mapping
$s^{(\beta)}\ni(a_j)_{j\in{\mathbb{Z}}^d}\leftrightarrow\sum\limits_{j\in{\mathbb{Z}}^d}a_j\e^{j\xi\sqrt{-1}}\in\mathcal{P}
(-\pi,\pi)$ (see \cite{uldist}).

\section{Construction of $p$-frames}\label{sec:6}
\indent
Let ${\theta}$ be a smooth non negative function such that
${\theta}(x)=1$, $x\in[-\pi+\varepsilon,\pi-\varepsilon]$, for
$0<\varepsilon<\frac14$, and $\mbox{supp}\,
{\theta} \subseteq[-\pi,\pi]$. Let $
\phi_k(x)=\mathcal{F}^{-1}(\theta(\cdot+k\pi))(x)$,
$x\in{\mathbb{R}}$,$\;$ $k\in{\mathbb{Z}}$. We can divide every $\theta(\cdot+k\pi)$ with the sum $\sum_{k\in\mathbb{Z}}\theta(\cdot+k\pi)$ in order to obtain the partition of unity. By the Paley-Wiener theorem, we have that
 $\phi_k\in
W^1_\mu({\mathbb{R}})$, $k\in{\mathbb{Z}}$. We say
that set $\{\phi_{i_1},\phi_{i_2},\ldots,\phi_{i_r}\}$, $i_1<i_2<\cdots<i_r$, is a set of
$r$ successive functions if $i_n=i_1+(n-1)$, $n=2,\ldots,r$. Note that for every $\xi\in\mathbb{R}$ there exist $\xi_0\in(-\pi,\pi)$ and $k\in\mathbb{Z}$ such that $\xi=\xi_0+k\pi$.

Now, we consider the following three cases.

\noindent $1^\circ$ The case of two successive functions.

If $\Phi=(\phi_i,\phi_{i+1})^T$, $i\in{\mathbb{Z}}$,
then
 $\rank[\widehat{\Phi}(\xi+2j\pi)_{j\in{\mathbb{Z}}}]$, $\xi\in{\mathbb{R}}$, is not a constant function on $\mathbb{R}$.
In this case, for the matrix $[\widehat{\Phi}(\xi+2j\pi)_{j\in{\mathbb{Z}}}]$, we obtain the $2\times \infty$ matrix
\[A(\xi_0)=\left[\begin{array}{llllll}
\cdots  &0 &\alpha^{\xi_0}_0     &0                         &0  \cdots  \\
\cdots  &0 &\alpha^{\xi_0}_{-1}  &\alpha^{\xi_0}_{1}        &0   \cdots
\end{array}\right],\] which depends on
$\xi_0\in(-\pi,\pi)$,  where $\alpha^{\xi_0}_{-1}=\theta(\xi_0-\pi)$, $\alpha^{\xi_0}_0=\theta(\xi_0)$ and $\alpha^{\xi_0}_1=\theta(\xi_0+\pi)$.

 For $\xi_0^1=\frac{\pi}{2}$, we have $\alpha^{\xi_0^1}_0\neq0$, $\alpha^{\xi_0^1}_{-1}\neq0$, and for $\xi_0^2=-\frac{\pi}{2}$, we have $\alpha^{\xi_0^2}_0\neq0$, $\alpha^{\xi_0^2}_1\neq0$. Since $\rank A(\xi_0^1)=1$ and $\rank A(\xi_0^2)=2$, we conclude that for successive functions $\phi_i,\phi_{i+1}$, $i\in\mathbb{Z}$, the rank of the matrix $[\widehat{\Phi}
(\xi+2j\pi)_{j\in{\mathbb{Z}}}]$ is not a constant function on $\mathbb{R}$.

\noindent $2^\circ$ The case of three successive functions.

If $\Phi=(\phi_i,\phi_{i+1},\phi_{i+2})^T$,
$i\in{\mathbb{Z}}$,
 then $\rank[\widehat{\Phi}(\xi+2j\pi)_{j\in{\mathbb{Z}}}]$
 is a constant function on ${\mathbb{R}}$. We have that
$\rank[\widehat{\Phi}
(\xi+2j\pi)_{j\in{\mathbb{Z}}}]=2$, for all $\xi\in\mathbb{R}$.

Indeed, the matrix $[\widehat{\Phi}(\xi+2j\pi)_{j\in{\mathbb{Z}}}]$, $\xi\in\mathbb{R}$, is $3\times \infty$ matrix
\[B(\xi_0)=\left[\begin{array}{llllll}
\cdots  &0 &\alpha^{\xi_0}_0       &0                  &0 &\cdots  \\
\cdots  &0 &\alpha^{\xi_0}_{-1}    &\alpha^{\xi_0}_{1} &0 &\cdots \\
\cdots  &0 &0                      &\alpha^{\xi_0}_{0} &0 &\cdots
\end{array}\right],\]  which depends on
$\xi_0\in(-\pi,\pi)$,  where $\alpha^{\xi_0}_{-1}=\theta(\xi_0-\pi)$, $\alpha^{\xi_0}_0=\theta(\xi_0)$ and $\alpha^{\xi_0}_1=\theta(\xi_0+\pi)$.
Since, $\theta(\xi_0)\neq0$ for all $\xi_0\in(-\pi,\pi)$, the matrix $B(\xi_0)$ has $2$ columns  with non-zero elements for all $\xi_0\in(-\pi,\pi) $. So, $\rank[\widehat{\Phi}
(\xi+2j\pi)_{j\in{\mathbb{Z}}}]$ is a constant function on $\mathbb{R}$ and $\rank[\widehat{\Phi}
(\xi+2j\pi)_{j\in{\mathbb{Z}}}]=2$, for all $\xi\in\mathbb{R}$.

$3^\circ$ The case of $r>3$ successive functions.

By taking $r+1$ successive functions
$\phi_i,\phi_{i+1},\ldots,\phi_{i+r}$, $r>2$, we have different situations
described in the next lemma.
\begin{lem}\label{zaparne}
$a)$ \ If $\Phi=(\phi_i,\phi_{i+1},\ldots,\phi_{i+r})^T$, for
$i\in{\mathbb{Z}}$, $r\in2\mathbb{N}+1$,   then
$\rank[\widehat{\Phi}(\xi+2j\pi)_{j\in{\mathbb{Z}}}]$ is not a constant
function on ${\mathbb{R}}$.

$b)$ \ If $\Phi=(\phi_i,\phi_{i+1},\ldots,\phi_{i+r})^T$, $i\in{\mathbb{Z}}$,
$r\in2{\mathbb{N}}$,   then
$\rank[\widehat{\Phi}(\xi+2j\pi)_{j\in{\mathbb{Z}}}]$ is a constant
function on ${\mathbb{R}}$ and  we have, for all $\xi\in\mathbb{R}$, and $r=2n$, $n\in\mathbb{N}$,
$\rank[\widehat{\Phi}
(\xi+2j\pi)_{j\in{\mathbb{Z}}}]=n+1$.
\end{lem}
\begin{proof}
Since supports of products
$\widehat{\phi}_{i_{1}}(\xi+2j_1\pi)\widehat{\phi}_{i_{2}}(\xi+2j_2\pi)$
are non-empty if the arguments are of the form $\xi-\pi$, $\xi$, $\xi+\pi$, modulo $2j\pi$, $j\in\mathbb{Z}$, we have
that only blocks with elements \[\left[\begin{array}{ll}
\theta(\xi)\;\;&\theta(\xi+2\pi)\\\theta(\xi-\pi)\;\;&\theta(\xi+\pi)\end{array}\right]\quad\mbox{or}\quad\left[\begin{array}{ll}
\theta(\xi-\pi)\;\;&\theta(\xi+\pi)\\\theta(\xi-2\pi)\;\;&\theta(\xi )\end{array}\right],\] can
determine the rank of the matrix $[\widehat{\Phi}(\xi+2j\pi)_{j\in{\mathbb{Z}}}]$. For any other choice of $2\times 2$
matrix, we get determinant equal $0$.

$(a)$ Let $\Phi=(\phi_i,\phi_{i+1},\ldots,\phi_{i+(2n-1)})^T$,
$n\in\mathbb{N}$.

For the matrix $[\widehat{\Phi}(\xi+2j\pi)_{j\in{\mathbb{Z}}}]$, we obtain the $r\times \infty$ matrix
\[A_r(\xi_0)=\left[\begin{array}{llllllllll}
\cdots   &\alpha^{\xi_0}_0     &0                             &0 &0&\cdots     &0                 &0               &\cdots\\
\cdots  &\alpha^{\xi_0}_{-1} &\alpha^{\xi_0}_{1}             &0 &0&\cdots     &0                 &0              &\cdots\\
\cdots   &0                &\alpha^{\xi_0}_{0}                    &0&0&\cdots      &0                 &0               &\cdots\\
\cdots   &0                &\alpha^{\xi_0}_{-1}&\alpha^{\xi_0}_{1} &0&\cdots     &0                &0                &\cdots\\
\cdots   &0                &0&\alpha^{\xi_0}_{0} &0&\cdots     &0                &0                &\cdots\\
\vdots &\vdots              &\vdots                       &\vdots  &\vdots&\cdots\;   &\vdots                         &\vdots&\vdots \\
\cdots   &0                 &0               &0                &0 &\cdots    &\alpha^{\xi_0}_{0}\;      &0               &\cdots\\
\cdots   &0                &0                &0               &0 &\cdots     & \alpha^{\xi_0}_{-1}\; &\alpha^{\xi_0}_{1}  &\cdots
\end{array}\right],\]
where $\alpha^{\xi_0}_{-1}=\theta(\xi_0-\pi)$, $\alpha^{\xi_0}_0=\theta(\xi_0)$ and $\alpha^{\xi_0}_1=\theta(\xi_0+\pi)$, $\xi_0\in(-\pi,\pi)$.

For $\xi_0^1=\frac{\pi}{2}$, we have $\alpha^{\xi_0^1}_0\neq0$, $\alpha^{\xi_0^1}_{-1}\neq0$, and for $\xi_0^2=-\frac{\pi}{2}$,  we obtain $\alpha^{\xi_0^2}_0\neq0$, $\alpha^{\xi_0^2}_1\neq0$. Since $\rank A_r(\xi_0^1)=n$ and $\rank A_r(\xi_0^2)=n+1$, we conclude that for even number of successive functions $\phi_i,\phi_{i+1},\ldots,\phi_{i+(2n-1)}$, $i\in{\mathbb{Z}}$, $n\in\mathbb{N}$,  the rank of the matrix $[\widehat{\Phi}
(\xi+2j\pi)_{j\in{\mathbb{Z}}}]$ is not a constant function on $\mathbb{R}$.

$(b)$ Let $\Phi=(\phi_i,\phi_{i+1},\ldots,\phi_{i+2n})^T$,
$i\in {\mathbb{Z}} $, $n\in\mathbb{N}$. The matrix
\[[\widehat{\Phi}(\xi+2j\pi)_{j\in{\mathbb{Z}}}]=\left[\begin{array}{llllllllll}
\cdots   &0   &\alpha^{\xi_0}_{0}   &0                  &0&0&\cdots      &0                 &0               &\cdots\\
\cdots  &0 &\alpha^{\xi_0}_{-1} &\alpha^{\xi_0}_{1}             &0 &0&\cdots     &0                 &0              &\cdots\\
\cdots    &0&0                &\alpha^{\xi_0}_{0}                    &0&0&\cdots      &0                 &0               &\cdots\\
\cdots    &0&0                &\alpha^{\xi_0}_{-1}&\alpha^{\xi_0}_{1} &0&\cdots     &0                &0                &\cdots\\
\vdots &\vdots              &\vdots                       &\vdots  &\vdots&\cdots\;   &\vdots                         &\vdots&\vdots \\
\cdots    &0&0                &0                &0               &0 &\cdots     & \alpha^{\xi_0}_{-1}\; &\alpha^{\xi_0}_{1}  &\cdots\\
\cdots    &0&0                 &0               &0                &0 &\cdots    &0 &\alpha^{\xi_0}_{0}\;                    &\cdots
\end{array}\right],\]  has the constant rank on $\mathbb{R}$. Indeed, since $\alpha^{\xi_0}_{0}\neq0$ for all $\xi_0\in(-\pi,\pi)$, the matrix $[\widehat{\Phi}
(\xi+2j\pi)_{j\in{\mathbb{Z}}}]$ has $n+1$ columns  with non-zero elements for all $\xi\in\mathbb{R}$ and $\rank[\widehat{\Phi}
(\xi+2j\pi)_{j\in{\mathbb{Z}}}]=n+1$, for all $\xi\in\mathbb{R}$.

\end{proof}

As a consequence of Corollary \ref{main} and Lemma \ref{zaparne},
$1^\circ$ we have the next result.
\begin{thm}
Let $\Phi=(\phi_i,\phi_{i+1},\ldots,\phi_{i+2n})^T$, for $i\in\mathbb{Z}$,
$n\in{\mathbb{N}}$. Then $V^p_\mu(\Phi)$ is closed in $L^p_\mu$,
for any $p\in[1,\infty]$, and $\{\phi_{i+s}(\cdot -j)\mid
j\in{\mathbb{Z}},\;0\le s\le 2n\}$  is a $p$-frame for
$V^p_\mu(\Phi)$ for any $p\in[1,\infty]$.
\end{thm}

\begin{rem} In this way we obtain the sequence of closed spaces $V^p_\mu(\phi_0,\phi_1,\phi_2)$, $V^p_\mu(\phi_0,\phi_1,\phi_2,\phi_3,\phi_4)$, $V^p_\mu(\phi_0,\phi_2,\ldots,\phi_6)$, etc. We also conclude that spaces generated with even numbers of successive functions, for example $V^p_\mu(\phi_0,\phi_1)$, $V^p_\mu(\phi_0,\phi_1,\ldots,\phi_5)$, are not closed subspaces of $L^p_\mu$.
\end{rem}

\begin{thm}\label{zanesusedne}
Let $\Phi=(\phi_{k_1},\phi_{k_2},\ldots,\phi_{k_r})^T$, $k_1<k_2<\cdots<k_r$, $r\in\mathbb{N}$, $k_1,k_2,\ldots,k_r\in\mathbb{Z}$, and
$V^p_{\mu,k_1,k_2,\ldots,k_r}=V^p_\mu(\Phi)$. We consider the following cases.
\begin{itemize}
\item[$i)$] \ $k_{i+1}-k_{i}>1$,\;$i=1,\ldots,r-1$;
\item[$ii)$] \ If for some $i_0\in\{1,2,\ldots,r\}$ holds $k_{i_0+1}-k_{i_0}=1$, then there exists $n\in\mathbb{N}$, $2\leq 2n\leq r $, such that $k_{i_0}+2$, $k_{i_0}+3$,\ldots, $k_{i_0}+2n$ are elements of the set $\{k_{1},\ldots,k_{r}\}$.
\end{itemize}
In these cases the following statements hold.
\begin{itemize}
\item[$1^\circ$] \
$\rank[\widehat{\Phi}(\xi+2j\pi)_{j\in{\mathbb{Z}}}]$ is a constant function for all
$\xi\in{\mathbb{R}}$.

\item[$2^\circ$] \ $V^p_\mu(\Phi)$ is closed in $L^p_\mu$ for any
$p\in[1,\infty]$.

\item[$3^\circ$] \ $\{\phi_{k_i}(\cdot -j)\mid
j\in{\mathbb{Z}},1\leq  i\leq  r\}$ is a $p$-frame for
$V^p_\mu(\Phi)$ for any $p\in[1,\infty]$.
\end{itemize}
\end{thm}
\begin{rem} $(1)$ We refer to \cite{2ald}  and \cite{xs} for the
$\gamma$-dense set $X=\{x_j\mid j\in J\}$. Let
$\phi_k(x)=\mathcal{F}^{-1}(\theta(\cdot-k\pi))(x)$,
$x\in{\mathbb{R}}$. Following the notation of \cite{xs}, we put
$\psi_{x_j}=\phi_{x_j}$ where $\{x_j\mid j\in J\}$ is
$\gamma$-dense set determined by $f\in
V^2(\phi)=V^2(\mathcal{F}^{-1}(\theta))$. Checking the proofs of
Theorems 3.1, 3.2 and 4.1 in \cite{xs}, we obtain the same
conclusions as in these theorems. These theorems show the
conditions and explicit $C_p$ and $c_p$ such that the inequality
\[c_p\|f\|_{L^p_\mu}\le \Big(\sum\limits_{j\in J}|\langle
f,\psi_{x_j}\rangle\mu(x_j)|^p\Big)^{1/p}\le
C_p\|f\|_{L^p_\mu}\] holds. This inequality guarantee the
feasibility of a stable and continuous reconstruction algorithm in
the signal spaces $V^p_\mu(\Phi)$.

$(2)$  Since the
 spectrum of the Gram matrix
 $[\widehat{\Phi},\widehat{\Phi}](\xi)$, for $\Phi$ defined in Theorem \ref{zanesusedne},  is bounded and bounded
 away from zero  (see \cite{bor}), then the family
 $\{\Phi(\cdot-j)\mid j\in{\mathbb{Z}}\}$
 forms a $p$-Riesz basis for $V^p_{\mu}(\Phi)$.

 $(3)$  For the appropriate choice of
 function $\Phi$, for example $\Phi$ defined in Theorem \ref{zanesusedne}, the associated Gram
 matrix satisfies a suitable Munckenhoupt $A_2$ condition
 (see \cite{nie}), so the system $\{\Phi(\cdot-j)\mid j\in{\mathbb{Z}}\}$ is
 stable in $L^2_{\mu}({\mathbb{R}})$.

$(4)$ Frames of the above type may be useful in
 applications since they satisfy assumptions of Theorem $3.1$ and
 Theorem $3.2$ in \cite{akr}. They show  that error analysis for
 sampling and reconstruction can be tolerated, or that the
 sampling and reconstruction problem in shift-invariant space is
 robust with respect to appropriate set of functions
 $\phi_{k_1},\ldots,\phi_{k_r}$.

\end{rem}

\section*{Acknowledgment}
\indent

The authors are indebted to the referee for pointing out $\ell^p_\mu$-stability of the synthesis operator   which helped us to improve and simplify the proof of the main theorem and include Theorem~\ref{mala-main} in our manuscript. Also, we are grateful to the referee  for the additional useful literature suggested by him.

The authors were supported in part by the Serbian
Ministry of Science and Technological Developments (Project
174024).

\noindent$^{1}$ Department of Mathematics and Informatics,\\ Faculty of Science,\\ University of Novi Sad, \\
Trg Dositeja Obradovica 4,\\ 21000 Novi Sad,\\
 Serbia\\
E-mail:{\tt stevan.pilipovic@dmi.uns.ac.rs}
\medskip

\noindent$^{2}$ Department of Mathematics and Informatics,\\ Faculty of Science,\\ University of
Kragujevac,\\
 Radoja Domanovi\'ca 12,\\ 34000  Kragujevac,\\ Serbia\\
 E-mail: {\tt suzanasimic@kg.ac.rs}
\end{document}